\documentclass[a4paper, 11pt]{article}
\usepackage{amsmath,amssymb,mathrsfs,amsthm}
\usepackage{array,enumerate}
\usepackage{latexsym}
\usepackage{mathrsfs}
\usepackage{multicol}
\usepackage{amscd}
\usepackage[all]{xy}
\usepackage{graphicx}
\usepackage[svgnames]{xcolor}
\usepackage{tikz}
\usetikzlibrary{intersections}
\theoremstyle{definition}
\newtheorem{thm}{Theorem}[section]
\newtheorem{dfn}[thm]{Definition}
\newtheorem{note-dfn}[thm]{Notation-Definition}

\newtheorem{exam}[thm]{Example}
\newtheorem{prop}[thm]{Proposition}
\newtheorem{cor}[thm]{Corollary}
\newtheorem{lem}[thm]{Lemma}
\newtheorem{rem}[thm]{Remark}

\newtheorem{prop-dfn}[thm]{Proposition-Definition}

\newcommand{\m}{\mathfrak{m}}
\newcommand{\N}{\mathbb{N}}
\newcommand{\Z}{\mathbb{Z}}
\newcommand{\Q}{\mathbb{Q}}

\newcommand{\F}{\mathbb{F}}

\newcommand{\Hom}{\mathrm{Hom}}
\newcommand{\mIm}{\mathrm{Im}}
\newcommand{\Spec}{\mathrm{Spec}\,}

\newcommand{\sep}{^{\mathrm{sep}}}
\title{On homotopy exact sequences for normal schemes}
\author{Ippei Nagamachi}
\date{}

\begin{document}

\maketitle

\begin{abstract}
Let $f : X \rightarrow S$ be a surjective morphism of finite type between connected locally Noetherian normal schemes.
We discuss sufficient conditions that the sequence of the \'etale fundamental groups
$$\pi_{1}(X\times_{S}\overline{\eta},\ast) \rightarrow \pi_{1}(X,\ast) \rightarrow \pi_{1}(S,\ast)\rightarrow 1$$
is exact, where $\overline{\eta}$ is a geometric generic point of $S$ and $\ast$ is a geometric point of $X\times_{S}\overline{\eta}$.
In the present paper, we generalize those in \cite{SGA1}, \cite{Ho}, and \cite{Mit}.
We show that the conditions we give are also necessary conditions in the case where, for instance, $S$ is an affine smooth curve over a field of characteristic $0$.
\end{abstract}

\tableofcontents

\setcounter{section}{-1}

\section{Introduction}
Let $f : X \rightarrow S$ be a surjective morphism of finite type between connected locally Noetherian normal schemes, $\overline{\eta}$ a geometric generic point of $S$, and $\ast$ a geometric point of $X\times_{S}\overline{\eta}$.
Suppose that the scheme $X\times_{S}\overline{\eta}$ is connected.
Consider the sequence of the \'etale fundamental groups
\begin{equation}
\pi_{1}(X\times_{S}\overline{\eta},\ast) \rightarrow \pi_{1}(X,\ast) \rightarrow \pi_{1}(S,\ast)\rightarrow 1
\label{introexact}
\end{equation}
In \cite{SGA1}, the following proposition is proved:
\begin{prop}(\cite{SGA1} Exp.X Corollaire 1.4)
Suppose that $f$ is proper and flat with geometrically reduced fibers.
Moreover, suppose that $f_{\ast}O_{X}=O_{S}$.
Then the sequence (\ref{introexact}) is exact.
\label{SGAhom}
\end{prop}
Note that the scheme $S$ is not assumed to be normal in \cite{SGA1}.
This proposition has been improved by Hoshi \cite{Ho} and Mitsui \cite{Mit} (cf.\,Propositions \ref{Hoshi-exact} and \ref{Mitsui-exact}).
They discussed the case where the morphism $f$ has geometrically reduced fibers.

In the present paper, we discuss homotopy exact sequences without this assumption.
Our main result is as follows (see Theorem \ref{suff} for weak conditions):
\begin{thm}
Suppose that the following conditions are satisfied:
\begin{itemize}
\item The morphism $f$ is flat or the scheme $S$ is regular.
\item Let $s$ be a point of $S$ whose local ring is of dimension $1$.
Write $\xi_{1}, \ldots, \xi_{n}$ for the generic points of the scheme $f^{-1}(s)$, $e_{i}$ for the multiplicity of $\xi_{i}$, and $k(\xi_{i})$ (resp.\,$k(s)$) for the residual field of $\xi_{i}$ (resp.\,$s$).
Then $\mathrm{gcd}\,(e_{1}, \ldots, e_{n})=1$ and the algebraic closure of $k(s)$ in $k(\xi_{i})$ is separable for some $i$.
\end{itemize}
Then the sequence (\ref{introexact}) is exact.
\label{introthm}
\end{thm}

We cannot drop any assumption of Theorem \ref{suff} (cf.\,Section \ref{necsection}, Example \ref{curve(F)}, and Remark \ref{curve(F)rem}).
For instance, we have the following two propositions (see Section \ref{necsection} for general settings):

\begin{prop} (cf.\,Corollary \ref{fundcor} and Example \ref{neceexam}.2)
Suppose that the scheme $S$ is the spectrum of a semi-local Dedekind domain which contains $\Q$, and that the scheme $X$ is regular.
Then the sequence (\ref{introexact}) is exact if and only if the greatest common divisor of the multiplicities of the irreducible components of each closed fiber of $f$ is $1$.
\label{introneceexam}
\end{prop}

\begin{prop} (cf.\,Proposition \ref{geomexact})
Suppose that the scheme $S$ is a smooth curve over a field $k$ of characteristic $0$, and the scheme $X$ is regular.
Moreover, suppose that the scheme $S$ is not proper rational (cf.\,Definition \ref{curvedfn}).
Then the sequence (\ref{introexact}) is exact if and only if the greatest common divisor of the multiplicities of the irreducible components of each closed fiber of $f$ is $1$.
\label{introgeomexact}
\end{prop}

We apply the above results to the case where $f: X \rightarrow S$ is a morphism from a regular variety to a hyperbolic curve (cf.\,Definition \ref{curvedfn}).
In particular, we prove that a certain morphism is characterized by the property that the kernel of the induced homomorphism between the \'etale fundamental groups is topologically finitely generated
(see Theorem \ref{curve criterion} for more details):
\begin{thm} (cf.\,Theorem \ref{curve criterion})
Suppose that $S$ is a hyperbolic curve over a field of characteristic $0$ and the scheme $X$ is regular.
The following three conditions are equivalent:
\begin{enumerate}
\item The greatest common divisor of the multiplicities of the irreducible components of each closed fiber of $f$ is $1$.
\item The sequence (\ref{introexact}) is exact.
\item The profinite group $\mathrm{Ker}(\pi_{1}(X, \ast)\rightarrow \pi_{1}(S, \ast))$ is topologically finitely generated.
\end{enumerate}
\label{intro curve criterion}
\end{thm}

Note that condition 1 is stated only in the language of schemes and that condition 3 is stated only in the language of topological groups.
Such statements are natural in the framework of anabelian geometry (cf.\,\cite{tama1}, \cite{Moch1}, \cite{Ho}).
In anabelian geometry, we attempt to get information of varieties from their \'etale fundamental groups.
In this sense, Theorem \ref{intro curve criterion} may be regarded as a group theoretical characterization of a morphism as written in condition 1.

Let us explain the strategies of the proofs of the exactness of the homotopy exact sequences given in \cite{SGA1}, \cite{Ho}, \cite{Mit}, and the present paper.
Let $X'\rightarrow X$ be an \'etale covering space whose pull-back $X'\times_{S}\overline{\eta}\rightarrow X\times_{S}\overline{\eta}$ has a section.
To show that the sequence (\ref{introexact}) is exact, we need to construct an \'etale covering space $S' \rightarrow S$ such that the pull-back $X\times_{S}S'$ is isomorphic to $X'$ over $X$.
In \cite{SGA1}, the Stein factorization of the morphism $X'\rightarrow S$ plays the role of $S'$.
In \cite{Ho}, where $f$ is not always proper, the normalization of $S$ in the function field of $X'$ plays the role of $S'$ there.
(In the present paper, we need to use the normalization of $S$ in the separable closure of the function field of $S$ in the function field of $X'$).
In \cite{Ho} and \cite{Mit}, they replace $X$ by another scheme over $S$ which is faithfully flat with geometrically normal fibers to show that the morphism $S'\rightarrow S$ is \'etale.
In our situation, we cannot find such a good scheme.
If the scheme $S$ is regular, it suffices to show that the morphism $S'\rightarrow S$ is \'etale over an open subscheme of $S$ whose complement is of codimension $\geq2$ by the Zariski-Nagata purity theorem.
If the scheme  $S$ is not regular, we need to assume that the morphism $f$ is flat (cf.\,Example \ref{norreg}.1).
In this case, we use Serre's criterion for normality to compare the morphism $X' \rightarrow X$ and the morphism $S' \rightarrow S$.

The content of each section is as follows: 
In Section \ref{suffsec}, we give the proof of Theorem \ref{introthm}.
In Section \ref{Lemdede}, we discuss properties of Dedekind schemes to have many tame extensions.
In Section \ref{necsection}, we give the proofs of Proposition \ref{introneceexam} and Proposition \ref{introgeomexact}.
In Section \ref{curves}, we give the proof of Theorem \ref{curve criterion}.
In Section \ref{app}, we discuss property (F).
In Section \ref{app2}, we discuss the homotopy exact sequence for geometrically connected  (not necessarily generic) fibers.

{\it Acknowledgements:} The author would like to thank Yuichiro Hoshi for some helpful discussions.
 Also, the author would like to thank Takeshi Tsuji for useful advice.
This work was supported by the Research Institute for Mathematical Sciences, a Joint Usage/Research Center located in Kyoto University.

\section{Sufficient conditions}
In this section, we give the proof of Theorem \ref{introthm} in a generalized setting.

Let $f: X \rightarrow S$ be a surjective morphism of locally finite type between connected  locally Noetherian normal separated schemes.
We write $K(X)$ (resp.\,$K(S)$) for the function field of $X$ (resp.\,$S$).
Take a geometric generic point $\overline{\eta}$ of $S$ and write $X_{\overline{\eta}}$ for the scheme $X\times_{S}\overline{\eta}$.
Suppose that $X_{\overline{\eta}}$ is connected (and hence irreducible).
Take a geometric point $\overline{x}$ of $X_{\overline{\eta}}$.
Then we obtain the following sequence of the \'etale fundamental groups:
\begin{equation}
\pi_{1}(X_{\overline{\eta}}, \overline{x})\rightarrow \pi_{1}(X, \overline{x}) \rightarrow \pi_{1}(S, \overline{x}) \rightarrow 1.
\label{exac}
\end{equation}

\begin{rem}
\begin{enumerate}
\item
Since $f$ is generically geometrically connected, the homomorphism $\pi_{1}(X, \overline{x}) \rightarrow \pi_{1}(S, \overline{x})$ is surjective by \cite{Ho} Lemma 1.6.
\item
Let $S' \rightarrow S$ be a finite \'etale morphism which the morphism $\overline{\eta} \rightarrow S$ factors through.
By Remark \ref{twoexact}.1, the sequence (\ref{exac}) is exact if and only if the sequence
$$\pi_{1}(X_{\overline{\eta}}, \overline{x})\rightarrow \pi_{1}(X\times_{S}S', \overline{x}) \rightarrow \pi_{1}(S', \overline{x}) \rightarrow 1$$
is exact.
\item
The composite homomorphism $\pi_{1}(X_{\overline{\eta}}, \overline{x})\rightarrow \pi_{1}(X, \overline{x})\rightarrow \pi_{1}(S, \overline{x})$ is trivial.
\item
Thus, the sequence (\ref{exac}) is exact if and only if
$$\mathrm{Im}(\pi_{1}(X_{\overline{\eta}}, \overline{x})\rightarrow \pi_{1}(X, \overline{x}))\supset \mathrm{Ker}(\pi_{1}(X, \overline{x}) \rightarrow \pi_{1}(S, \overline{x})  ).$$
\end{enumerate}
\label{twoexact}
\end{rem}

First, we recall sufficient conditions given by Hoshi and Mitsui which generalize Proposition \ref{SGAhom}.

\begin{prop}(\cite{Ho} Proposition 1.10)
Suppose that there exist a connected locally Noetherian normal separated scheme $Y$ and a morphism $p: Y \rightarrow X$.
Moreover, suppose that the following conditions are satisfied:
\begin{itemize}
\item The morphism $p$ is dominant and induces an outer surjection\\ $\pi_{1}(Y) \rightarrow \pi_{1}(X)$.
\item The morphism $f$ is generically geometrically integral. 
\item The composite morphism $f\circ p$ is of finite type, faithfully flat, geometrically normal, and generically geometrically connected.
\end{itemize}
Then the sequence (\ref{exac}) is exact.
\label{Hoshi-exact}
\end{prop}

\begin{prop}(\cite{Mit} Theorem 4.22)
Suppose that $f$ is flat and geometrically reduced.
Moreover, suppose that the sheaf $O_{S}$ is integrally closed in the sheaf $f_{\ast}O_{X}$.
Then the sequence (\ref{exac}) is exact.
\label{Mitsui-exact}
\end{prop}

Since the schemes $X$ and $S$ are normal, these schemes enjoy the following properties:
\begin{lem}
Let $U$ be a connected locally Noetherian normal scheme.
Write $K(U)$ for the function field of $U$.
\begin{enumerate}
\item Let $\ast$ be a geometric point of $\Spec K(U)$.
Then the homomorphism $\pi_{1}(\Spec K(U), \ast) \rightarrow \pi_{1}(U, \ast)$ is surjective.
\item Let $V \rightarrow U$ be a connected \'etale covering space.
Write $K(V)$ for the function field of $V$.
Let $L$ be an intermediate field of $K(U)\subset K(V)$.
Then the normalization of $U$ in $L$ is an \'etale covering space of $U$.
\end{enumerate}
\label{interfield}
\end{lem}

\begin{proof}
Assertion 1 is well-known.
Since the $\pi_{1}(\Spec K(U), \ast)$-equivariant morphism
$$(\Hom_{U}(\ast,V)\simeq)\Hom_{\Spec K(U)}(\ast,\Spec K(V))\to\Hom_{\Spec K(U)}(\ast,\Spec L)$$
is surjective, there exists a connected \'etale covering space $W\to U$ such that $W\times_{U}\Spec K(U)$ is isomorphic to $\Spec L$ over $K(U)$.
Since $U$ is normal and the morphism $W\to U$ is \'etale, $W$ is also normal.
Since $W$ is finite over $U$, assertion 2 holds.
\end{proof}

We rephrase the exactness of the sequence (\ref{exac}) in terms of \'etale covering spaces of $X$.

\begin{prop}
The following four conditions are equivalent:
\begin{enumerate}
\item $\mathrm{Im}(\pi_{1}(X_{\overline{\eta}}, \overline{x})\rightarrow \pi_{1}(X, \overline{x})) \supset \mathrm{Ker}(\pi_{1}(X, \overline{x}) \rightarrow \pi_{1}(S, \overline{x})).$
\item
Let $C$ be a finite set with a continuous left $\pi_{1}(X, \overline{x})$-action.
Suppose that there exists a $\pi_{1}(X_{\overline{\eta}}, \overline{x})$-orbit of $C$ on which $\pi_{1}(X_{\overline{\eta}}, \overline{x})$ acts trivially.
Then there exist a finite set $D$ with a continuous transitive left $\pi_{1}(S, \overline{x})$-action and a $\pi_{1}(X, \overline{x})$-equivariant isomorphism between $D$ and $C$.
\item Let $X'$ be a connected \'etale covering space of $X$.
Suppose that the \'etale covering space $X_{\overline{\eta}}\times_{X}X'\rightarrow X_{\overline{\eta}}$ has a section.
Then there exist an \'etale covering space $S' \rightarrow S$ and an $X$-isomorphism between $X\times_{S}S'$ and $X'$.
\item Let $X'$ be a connected \'etale covering space of $X$.
Write $K_{X'/S}$ for the separable closure of $K(S)$ in the function field of $X'$.
The normalization $N_{X'/S}$ of $S$ in $K_{X'/S}$ is \'etale over $S$.
\end{enumerate}
\label{essential}
\end{prop}

\begin{proof}
Since the homomorphism $\pi_{1}(X, \overline{x}) \rightarrow \pi_{1}(S, \overline{x})$ is surjective, condition 1 is equivalent to condition 2.
The equivalence of 2 and 3 is clear.

We prove the equivalence of 3 and 4.
Let $X'$ be a connected \'etale covering space of $X$.
Write $K(X')$ for the function field of $X'$, $K_{X'/S}$ for the separable closure of $K(S)$ in $K(X')$, and $N_{X'/S}\rightarrow S$ for the normalization of $S$ in $K_{X'/S}$.

First, we prove the implication $3 \Rightarrow 4$.
The normalization $X_{N}$ of $X'$ in the composite field $K(X)K_{X'/S}$ in $K(X')$ is an \'etale covering space of $X$ by Lemma \ref{interfield}.2.
Moreover, since the morphism $X_{\overline{\eta}}\times_{X}X_{N}\to X_{\overline{\eta}}$ has a section, there exist a finite \'etale covering space $S' \rightarrow S$ and an $X$-isomorphism $X_{N} \simeq X\times_{S}S'$ by condition 3.
Therefore $N_{X'/S}$ is isomorphic to $S'$ over $S$ and thus condition 4 holds.

Next, we prove the implication $4 \Rightarrow 3$.
Suppose that the morphism $X_{\overline{\eta}}\times_{X}X' \rightarrow X_{\overline{\eta}}$ has a section.
By condition 4, $N_{X'/S}$ is an \'etale covering space of $S$.
It suffices to show that the induced morphism $\phi: X' \rightarrow X\times_{S}N_{X'/S}$ is an isomorphism.
Since $X\times_{S}N_{X'/S}$ is \'etale over $X$ and connected, the morphism $\phi$ is finite \'etale surjective.
The number of connected components of $X_{\overline{\eta}}\times_{S}N_{X'/S}$ coincides with the covering degree of $N_{X'/S}$ over $S$.
On the other hand, the number of connected components of $X_{\overline{\eta}}\times_{X}X'=\overline{\eta}\times_{S}X'$ coincides with the extension degree $[K_{X'/S}:K(S)]$.
Therefore, there is a bijection between the set of connected components of $X_{\overline{\eta}}\times_{X}X'$ and that of $X_{\overline{\eta}}\times_{S}N_{X'/S}$.
Since the morphism $X_{\overline{\eta}}\times_{X}X' \rightarrow X_{\overline{\eta}}$ has a section, we can show that the covering degree of $X'$ over $X\times_{S}N_{X'/S}$ is $1$.
Thus, condition 3 holds.
\end{proof}

\begin{rem}
Proposition \ref{essential} holds if $f$ is dominant (even if $f$ is not surjective).
\label{dominant}
\end{rem}

Recall that we do not assume that the scheme $S$ is regular.
Since we cannot use the Zariski-Nagata purity theorem, we show the following technical lemma needed later:

\begin{lem}
Let $S'$ be an integral scheme and $S' \rightarrow S$ a quasi-finite dominant morphism.
\begin{enumerate}
\item Suppose that $f$ is flat and the extension between the function fields of $S'$ and $S$ is separable.
Then the scheme $X\times_{S}S'$ is integral.
\item Moreover, suppose that the scheme $S'$ is normal and the morphism $S' \rightarrow S$ is \'etale over each point of $S$ whose local ring is of dimension $1$.
Then the scheme $X\times_{S}S'$ is normal.
\end{enumerate}
\label{nonZar}
\end{lem}

\begin{proof}
Write $K(S')$ for the function field of $S'$.
Since $f$ is generically geometrically connected and $K(S')$ is separable over $K(S)$, the scheme $X\times_{S}\Spec K(S')$ is integral.
Therefore, assertion 1 follows from flatness of $f$.

By Serre's criterion for normality, it suffices to show that the scheme $X\times_{S}S'$ satisfies ($R_{1}$) and ($S_{2}$) to prove assertion 2.
Any point of the scheme $X\times_{S}S'$ over a point of $S'$ whose local ring is of dimension $\leq1$ is normal by the assumption on the morphism $S' \rightarrow S$.
Since $f$ is flat, the image of any point of the scheme $X\times_{S}S'$ whose local ring is of dimension $1$ is a point of $S'$ whose local ring is of dimension $\leq1$.
Therefore, $X\times_{S}S'$ satisfies ($R_{1}$).
Since $f$ is flat and $S'$ is normal, any point of the scheme $X\times_{S}S'$ over a point of $S'$ whose local ring is of dimension $\geq2$ is of depth $\geq2$.
Therefore, the scheme $X\times_{S}S'$ satisfies ($S_{2}$).
\end{proof}

\begin{prop}
Suppose that $f$ is flat or $S$ is regular.
Then the four conditions in Proposition \ref{essential} are equivalent to the following condition:
\begin{enumerate}
\setcounter{enumi}{4}
\item Let $X'$ be a connected \'etale covering space of $X$.
Write $K_{X'/S}$ for the separable closure of $K(S)$ in the function field of $X'$.
The normalization $N_{X'/S}$ of $S$ in $K_{X'/S}$ is \'etale over each point of $S$ whose local ring is of dimension $1$.
\end{enumerate}
\label{essprop}
\end{prop}

\begin{proof}
The implication $4\Rightarrow 5$ is clear.
We prove the implication $5\Rightarrow 4$.
By condition 5, the morphism $N_{X'/S} \rightarrow S$ is \'etale over each point of $S$ whose local ring is of dimension $\leq 1$.
If $S$ is regular, the morphism $N_{X'/S} \rightarrow S$ is \'etale by the Zariski-Nagata purity theorem.
Hence we may assume that $f$ is flat.
Then the scheme $X\times_{S}N_{X'/S}$ is connected normal by Lemmas \ref{nonZar}.1 and \ref{nonZar}.2.
Since the scheme $X\times_{S}N_{X'/S}$ coincides with the normalization of $X$ in the composite field $K(X)K_{X'/S}$ in the function field of $X'$, the morphism $X\times_{S}N_{X'/S} \rightarrow X$ is \'etale by Lemma \ref{interfield}.2.
Since $f$ is faithfully flat, the morphism $X\times_{S}N_{X'/S} \rightarrow X$ is also \'etale.
We finish the proof of Proposition \ref{essprop}.
\end{proof}

\begin{dfn}
Let $\{\iota_{i}: k \hookrightarrow K_{i}\}_{i\in I}$ be a set of inclusions of fields.
We say that the inclusions $\{ \iota_{i} \}$ satisfy property (F) if the following condition is satisfied:\\
(F): For any algebraic separable extension $L_{i}$ of $K_{i} \;(i\in I$) and any subfield $l$ of the product ring $\prod_{i\in I}L_{i}$ which is algebraic over the diagonal subfield $k$ defined by $\iota_{i}$, the extension $k \subset l$ is separable.
\label{(F)}
\end{dfn}

\begin{rem}
\begin{enumerate}
\item If $K_{i}$ is geometrically reduced over $k$ for some $i$, the inclusions $\{ \iota_{i} \}$ satisfy property (F).
In fact, if $k$ is purely inseparably closed in $K_{i}$ (i.e., $k^{p^{-\infty}}\cap K_{i}=k$) for some $i$, the inclusions $\{ \iota_{i} \}$ satisfy property (F).
\item We discuss property (F) in Section \ref{app}.
\end{enumerate}
\label{(F)geom}
\end{rem}

\begin{dfn}
We say that $f$ satisfies property (R) if the following condition is satisfied.\\
(R): Let $s$ be a point of $S$ whose local ring is of dimension $1$.
Write $\xi_{i}\;(i\in I)$ for the generic points of the scheme $f^{-1}(s)$, $e_{i}$ for the multiplicity of $\xi_{i}$, and $k(\xi_{i})$ (resp.\,$k(s)$) for the residual field of $\xi$ (resp.\,$s$).
Then $\mathrm{gcd}_{i\in I}e_{i}=1$ and the inclusions $k(s) \hookrightarrow k(\xi_{i})\;(i\in I)$ satisfy property (F).
\end{dfn}

We prove the main theorem of the present paper (cf.\,Theorem \ref{introthm})

\begin{thm}
Suppose that $f$ satisfies property (R) and one of the following conditions is satisfied:
\begin{multicols}{2}
\begin{itemize}
\item The morphism $f$ is flat.
\item The scheme $S$ is regular.
\end{itemize}
\end{multicols}
Then the sequence (\ref{exac}) is exact.
\label{suff}
\end{thm}

\begin{proof}
By Remark \ref{twoexact}.4, it suffices to show that
$$\mathrm{Ker}(\pi_{1}(X, \overline{x}) \rightarrow \pi_{1}(S, \overline{x})) \subset \mathrm{Im}(\pi_{1}(X_{\overline{\eta}}, \overline{x})\rightarrow \pi_{1}(X, \overline{x})).$$
Let $X' \rightarrow X$ be a finite \'etale covering space.
Write $K_{X'/S}$ for the separable closure of $K(S)$ in the function field of $X'$.
By Proposition \ref{essprop}, it suffices to show that the normalization $N_{X'/S}$ of $S$ in $K_{X'/S}$ is finite \'etale over $S$ at each point of $N_{X'/S}$ whose local ring is of dimension $1$.
Let $n$ be such a point of $N_{X'/S}$.
Write $s$ for the image of $n$ in $S$.
It suffices to show that the extension of the discrete valuation rings $O_{S,s} \subset O_{N_{X'/S},n}$ is unramified.
Therefore, the assertion follows from the hypothesis of Theorem \ref{suff}. 
\end{proof}

\begin{rem}
If the morphism $f$ is not flat and the scheme $S$ is not regular, Theorem \ref{suff} does not hold in general (cf.\,Example \ref{norreg}.1).
\end{rem}

\label{suffsec}

\section{Lemmas for Dedekind schemes}
In this section, we discuss some properties of Dedekind schemes which we use in Section \ref{necsection}.
\label{Lemdede}
\subsection{A fundamental lemma}
We prove a lemma for Dedekind schemes needed later.

\begin{lem}
Let $R$ be a strictly henselian discrete valuation ring.
Write $K$ for the field of fractions of $R$.
Let $K'$ be a finite tamely ramified extension field of $K$.
Write $R'$ for the normalization of $R$ in $K'$ and $e'$ for the ramification index of the extension $R'/R$.
Let $A$ be a discrete valuation ring which dominates $R$ such that the ramification index of the extension $A/R$ is $e$.
Suppose that the field of fractions $L$ of $A$ is geometrically connected over $K$ and $e$ is divisible by $e'$.
Then the normalization $A'$ of $A\otimes_{R}R'$ (cf.\,Lemma \ref{nonZar}.1) is \'etale over $A$.
\label{essnec}
\end{lem}

\begin{proof}
Let $\widetilde{A}$ be a strict henselization of $A$.
Then $\widetilde{A}\otimes_{A}A'$ is the normalization of $\widetilde{A}\otimes_{R}R'$ (in its total ring of fractions).
Therefore, it suffices to show that $\widetilde{A}\otimes_{A}A'$ is the product ring of $e'$ copies of $\widetilde{A}$.
Let $\varpi$ (resp.\,$\varpi '$; $\varpi_{\widetilde{A}}$) be a uniformizer of $R$ (resp.\,$R'$; $\widetilde{A}$).
There exists a unit $u'$ (resp.\,$u_{\widetilde{A}}$)  of $R'$ (resp.\,$\widetilde{A}$) such that $\varpi =u'(\varpi ')^{e'}$ (resp.\,$\varpi =u_{\widetilde{A}}(\varpi_{\widetilde{A}})^{e}$).
Since there exists a unit $v'$ (resp.\,$v_{\widetilde{A}}$) of $R'$ (resp.\,$\widetilde{A}$) which satisfies that $(v')^{e'}=u'$ (resp.\,$v_{\widetilde{A}}^{e'}=u_{\widetilde{A}}$), we may assume that $(\varpi ')^{e'}=\varpi$ (resp.\,$(\varpi_{\widetilde{A}})^{e}=\varpi$).
Thus, $R'$ is isomorphic to $R[T]/(T^{e'}-\varpi)$ and $\widetilde{A}\otimes_{R}R'$ is isomorphic to $\widetilde{A}[T]/\underset{1\leq i \leq e'}{\prod}(T-\zeta_{e'}^{i}(\varpi_{\widetilde{A}})^{\frac{e}{e'}})$.
Here, $\zeta_{e'}$ is a primitive $e'$-th root of unity in $\widetilde{A}$.
Therefore, $\widetilde{A}\otimes_{A}A'$ is isomorphic to $\underset{1\leq i \leq e'}{\prod}\widetilde{A}[T]/(T-\zeta_{e'}^{i}(\varpi_{\widetilde{A}})^{\frac{e}{e'}})$.
We finished the proof of Lemma \ref{essnec}.
\end{proof}

\subsection{Examples of Dedekind schemes}
We discuss whether there exists a convenient (cf.\,Definition \ref{(T)dfn}) tame covering space of a given Dedekind scheme.

\begin{dfn}
\begin{enumerate}
\item Let $S$ be a scheme.
We shall say that $S$ is a \textit{Dedekind scheme} if $S$ is connected, locally Noetherian, normal, and of dimension $1$.
\item Let $S$ be a Dedekind scheme.
We say that $S$ has property (T) if, for each closed point $s\in S$ and a prime number $l$ which is not divisible by the characteristic of the residual field of $s$, there exist a normal scheme $S'$ and a finite dominant morphism $S'\to S$ which satisfy the following conditions:
\begin{itemize}
\item The morphism $S' \to S$ is finite Galois \'etale over $S\setminus\{s\}$.
\item For any closed point $s'$ over $s$, the ramification index of $S'\to S$ at $s'$ (, which independent of the choice of $s'$ since $S'$ is Galois over $S$,) is $l$.
\end{itemize}
\end{enumerate}
\label{(T)dfn}
\end{dfn}

\begin{rem}
The conditions on $S'$ in Definition \ref{(T)dfn}.2 is equivalent to the following conditions:
\begin{itemize}
\item
The ramification index of $S'\to S$ at each point of $S'$ over $s$ is equal to $1$ or $l$.
\item
There exists a point of $S'$ over $s$ where the ramification index of the morphism $S'\to S$ is equal to $l$.
\item The morphism $S' \to S$ is finite \'etale over $S\setminus\{s\}$.
\end{itemize}
\label{weak}
\end{rem}

\begin{lem}
Let $R$ be a semi-local Dedekind domain (hence a principal ideal domain).
Then the scheme $S=\Spec R$ satisfies property (T).
\label{semi}
\end{lem}

\begin{proof}
Write $K$ for the field of fractions of $R$, $\m_{i} \;(1\leq i \leq n)$ for the maximal ideals of $R$, and $p_{i}\;(1\leq i \leq n)$ for the characteristic of $R/\m_{i}$.
Let $l$ be a prime number which is not divisible by $p_{1}$.
By Chinese Remainder Theorem, we can choose elements $a$ and $b$ of $R$ which satisfy the following conditions:
\begin{multicols}{2}
\begin{itemize}
\item $\begin{cases} a \in \m_{i} \quad ( l\notin \m_{i}) \\
a \equiv 1\, \mathrm{mod}\, \m_{i} \quad( l\in \m_{i})\end{cases}$
\item $\begin{cases} b \in \m_{1}\setminus \m_{1}^{2} \\
b \equiv 1\, \mathrm{mod}\, \m_{i} \quad(\m_{i}\neq\m_{1})\end{cases}.$
\end{itemize}
\end{multicols}
Then the extension field of $K$ defined by the polynomial $T^{l}-aT-b$ satisfies the conditions in Remark \ref{weak}.
Therefore, the Dedekind scheme $S$ satisfies property (T).
\end{proof}

\begin{dfn}
Let $k$ be a field and $C$ a scheme over $k$.
\begin{enumerate}
\item We say that $C$ is a \textit{smooth curve} over $k$ if the structure morphism $C \rightarrow \Spec k$ is separated, of finite type, smooth of relative dimension $1$, and geometrically connected.
Let $C$ be a smooth curve over $k$ and $\overline{k}$ an algebraic closure of $k$.
Write $\overline{C}$ for the regular compactification of $C$ over $k$, $g_{C}$ for the genus of $\overline{C}$, and $r_{C}$ for the number of closed points of the scheme $(\overline{C}\setminus C)\times_{\Spec k}\Spec \overline{k}$.
\item We say that  $C$ is \textit{proper rational} if $C$ is a smooth curve over $k$ and $g_{C}=r_{C}=0$.
\item We say that $C$ is a \textit{hyperbolic curve} if $C$ is a smooth curve over $k$, $2g_{C}+r_{C}-2>0$, and the reduced closed subscheme $\overline{C}\setminus C$ of $\overline{C}$ is finite \'etale over $\Spec k$.
\end{enumerate}
\label{curvedfn}
\end{dfn}

\begin{lem}
Let $k$ be a field and $S$ a smooth curve over $k$.
Suppose that $S$ is not proper rational.
Then the Dedekind scheme $S$ satisfies property (T).
\label{nonproper}
\end{lem}

\begin{proof}
We may assume that the field $k$ is algebraically closed.
Take $s\in S$ and $l$ as in Definition \ref{(T)dfn}.

First, suppose that $S$ is not proper.
Choose a point $s'$ of $\overline{S}\setminus S$, where $\overline{S}$ is the smooth compactification of $S$ over $k$.
Then there exists a finite dominant morphism $\overline{S'} \rightarrow \overline{S}$ from a proper smooth curve $\overline{S'}$ over $k$ to $\overline{S}$ which is a $\Z/l\Z$-\'etale covering space over $\overline{S}\setminus\{s,s'\}$ and totally (tamely) ramified over $s$ and $s'$.
Therefore, $S$ satisfies property (T).

Next, suppose that $S$ is proper (and hence the genus of $S$ is not $0$).
Then take a nontrivial finite Galois \'etale covering space $S' \rightarrow S$.
Choose two closed points $s_{1}$ and $s_{2}$ of $S'$ over $s$.
Then there exists a $\Z/l\Z$-Galois covering space $S'' \rightarrow S'$ which is finite \'etale over $S'\setminus\{s_{1}, s_{2}\}$ and totally ramified over $s_{1}$ and $s_{2}$.
By Remark \ref{weak}, $S$ satisfies property (T).
\end{proof}

\section{Necessary conditions}
In this section, we show that
the conditions in Theorem \ref{suff} are necessary for the sequence (\ref{exac}) to be exact in some cases.

Let $f, X, S, \overline{\eta}, X_{\overline{\eta}}$, and $\overline{x}$ be as in Section \ref{suffsec}.
In this section, we suppose that the morphism $f$ is generically geometrically reduced (cf.\,Remark \ref{geomred}.1) and the scheme $X$ is regular (cf.\,Example \ref{norreg}.2).
Moreover, suppose that the morphism $f$ is flat.
Since $f$ is faithfully flat, it follows that $S$ is also regular.

\begin{rem}
\begin{enumerate}
\item
We show that the condition that $f$ satisfies property (R) is not sufficient for the sequence (\ref{exac}) to be exact.
Suppose that $S$ is a smooth curve over an algebraically closed field $k$ of positive characteristic and $f$ satisfies property (R).
Note that the sequence of profinite groups (2) is exact by Theorem  \ref{suff}.
Write $F: S \to S$ for the absolute Frobenius morphism of $S$.
Then the composite morphism $F\circ f$ is flat and does not satisfy property (R).
For the morphism $F\circ f$, we can consider a sequence of the \'etale fundamental groups similar to (\ref{exac}), which is also exact since $F$ is a universally homeomorphism.
\item
Since $f$ is formally smooth over the generic point of $S$, there exists a dense open subset of $S$ such that $f$ is formally smooth there.
\end{enumerate}
\label{geomred}
\end{rem}

\begin{thm}
Suppose that there exist a connected normal scheme $S'$ and a finite dominant morphism $S'\rightarrow S$ which satisfy the following conditions:\\
\begin{itemize}
\item The morphism $S'\rightarrow S$ is \'etale over the generic point of $S$.
\item Let $s' $ be a point of $S'$ whose local ring is of dimension $1$.
Write $s$ for the image of $s'$ in $S$.
Then the extension of discrete valuation rings $O_{S,s}\subset O_{S',s'}$ is at most tamely ramified with ramification index $e_{s'}$.
\item Let $\xi$ be a generic point of the scheme $f^{-1}(s)$ and write $e$ for the multiplicity of $\xi$.
Then $e$ is divisible by $e_{s'}$.
\end{itemize}
Then the normalization $X'$ of the scheme $X\times_{S}S'$ in its function field is \'etale over $X$.
Moreover, the sequence (\ref{exac}) is not exact.
\label{essnece}
\end{thm}

\begin{proof}
Note that the scheme $X\times_{S}S'$ is integral by Lemma \ref{nonZar}.1.
The second assertion follows from the first assertion and Proposition  \ref{essential}.
To show the first assertion, we may assume that $S(=\Spec O_{S,s})$ is the spectrum of a discrete valuation ring by the discussion in the proof of Proposition 1.7.
Let $s$ be the closed point of $S$ and $\xi_{i}\:(i\in I)$ the generic points of $f^{-1}(s)$.
By the Zariski-Nagata purity theorem, it suffices to show that the morphism $X'\to X$ is \'etale at each $\xi_{i}$.
Write $O_{S,s}^{\mathrm{sh}}$ for the strict henselization of $O_{S,s}$.
Since the morphism $\Spec O_{S,s}^{\mathrm{sh}}\to S$ is faithfully flat, we may assume that $S=\Spec O_{S,s}^{\mathrm{sh}}$.
Therefore, we can show Theorem \ref{essnece} by applying Lemma \ref{essnec} to each $O_{X,\xi_{i}}$.
\end{proof}

\begin{rem}
We use the same notation in Theorem \ref{essnece}.
Suppose that the scheme $S$ is quasi-compact (hence Noetherian) and the morphism $f$ is of finite type.
Then the set of points $s$ of $S$ satisfying the following properties is finite by Remark \ref{geomred}.2:
\begin{itemize}
\item
$\dim O_{S,s}=1$.
\item
There exists a point $s'$ of $S'$ over $s$ satisfying $e_{s'}\neq1$.
\end{itemize}
\end{rem}

\begin{cor}
Suppose that the following conditions are satisfied:
\begin{itemize}
\item $S$ is a Dedekind scheme and satisfies property (T) (cf.\,Definition \ref{(T)dfn}, Lemma \ref{semi}, Lemma \ref{nonproper}).
\item Let $s$ be a closed point of $S$ and $\xi_{i}\;(i\in I)$ generic points of the scheme $f^{-1}(s)$.
Write $e_{i}$ for the multiplicity of $\xi_{i}$, $k(\xi_{i})$ (resp.\,$k(s)$) for the residual field of $\xi$ (resp.\,$s$), and $p(s)$ for the characteristic of the field $k(s)$.
Then $(e_{s}:=)\,\gcd_{i\in I}e_{i}$ is not divisible by $p(s)$ and the inclusions $k(s) \hookrightarrow k(\xi_{i})\;(i\in I)$ satisfy property (F).
\end{itemize}

Then the sequence (\ref{exac}) is exact if and only if $e_{s}=1$ for each closed point $s$ of $S$.
\label{fundcor}
\end{cor}

\begin{proof}
Corollary \ref{fundcor} follows from Theorem \ref{suff} and Theorem \ref{essnece}.
\end{proof}

\begin{exam} (cf.\,Proposition \ref{introneceexam})
We discuss the conditions of Corollary \ref{fundcor}.
\begin{enumerate}
\item Suppose that $S$ is the spectrum of a discrete valuation ring with perfect residual field of characteristic $p$.
Then properties (T) and (F) are automatically satisfied (cf.\,Lemma \ref{semi}).
Therefore, we only need to suppose that $e_{s}$ is not divisible by $p$ to apply Corollary \ref{fundcor}.
\item
Suppose that $S$ is the spectrum of a semi-local Dedekind domain which contains $\Q$.
Then all the conditions of Corollary \ref{fundcor} are automatically satisfied (cf.\,Lemma \ref{semi}).
\end{enumerate}
\label{neceexam}
\end{exam}

\begin{prop} (cf.\,Proposition \ref{introgeomexact})
Suppose that $S$ is a smooth curve over a field $k$ of characteristic $0$.
Moreover, suppose that $S$ is not proper rational.
Then the sequence (\ref{exac}) is exact if and only if the greatest common divisor of the multiplicities of the irreducible components of each closed fiber of $f$ is $1$.
\label{geomexact}
\end{prop}

\begin{proof}
Proposition \ref{geomexact} follows from Lemma \ref{nonproper} and Corollary \ref{fundcor}.
\end{proof}

\begin{exam}
Let $k$ be an algebraically closed field, $C'$ a smooth curve over $k$, and $\sigma$ an automorphism of $C'$ over $k$ of prime order $l\,(>2)$.
Suppose that $\sigma$ has $n(>0)$ fixed points $c'_{1},\ldots,c'_{n}$.
Write $C'\rightarrow C$ for the quotient scheme of $C'$ by $\Z/l\Z=\langle\sigma \rangle$ and $c_{i}$ for the image of $c'_{i}$ in $C$ for $1\leq i \leq n$.
Then $\{(c'_{i}, c'_{j})\in C'\times_{\Spec k}C'\mid 1\leq i, j \leq n\}$ is the set of fixed points of $C'\times_{\Spec k}C'$ for the action of $\Z/l\Z=\langle(\sigma^{2}, \sigma)\rangle$.
Let $B'$ be the scheme obtained by the blow-ups of $C'\times_{\Spec k}C'$ at all such points.
The $\Z/l\Z$-action on $C'\times_{\Spec k}C'$ induces a natural $\Z/l\Z$-action on $B'$.
Then the scheme $B'$ has exactly $2n^{2}$ fixed points.
Write $Y'$ for the open subscheme of $B'$ whose complement is the set of the fixed points.
Let $Y\rightarrow B \rightarrow Z$ be the quotient morphisms of the morphisms $Y' \rightarrow B' \rightarrow C'\times_{\Spec k}C'$ by $\Z/l\Z=\langle (\sigma^{2}, \sigma) \rangle$.
Since $\{(c'_{i}, c'_{j}); 1\leq i, j \leq n\}$ is the ramified locus of the morphism $C'\times_{\Spec k}C' \rightarrow Z$, the scheme $Z$ is not regular by the Zariski-Nagata purity theorem. 
Note that the morphism $Y' \rightarrow Y$ is \'etale.
\begin{enumerate}
\item
We show that Theorem \ref{suff} does not hold in general if the morphism $f$ is not flat and the scheme $S$ is not regular.
Consider the case where $f$ is the morphism $Y\rightarrow Z$.
Note that $f$ is not flat since the dimensions of fibers of $f$ are not constant.
Moreover, $S(=Z)$ is not regular.
Since $f$ is birational and $Y'\to X(=Y)$ is \'etale, the normalization of $S(=Z)$ in the separable closure of the function field of $S$ in the function field of $Y'$ coincides with $C'\times_{\Spec k}C'$.
Since $C'\times_{\Spec k}C'$ is not \'etale over $S$, the sequence (\ref{exac}) is not exact by Proposition \ref{essential}.
\item
We show that Proposition \ref{geomexact} does not hold in general if $X$ is not regular.
Suppose that $C'\simeq\mathbb{A}_{k}=\Spec k[T]$ and $\sigma$ is induced by the $k$-algebra homomorphism determined by $T\mapsto \zeta_{l}T$, where $\zeta_{l}$ is a primitive $l$-th root of unity.
Then $n=1$ and $c'_{1}=0\in C'$
Note that the second projection $C'\times_{\Spec k}C' \rightarrow C'$ is a $\Z/l\Z$-equivariant morphism, which induces a morphism $Z\to C$.
Consider the case where $f$ is this morphism $Z\rightarrow C$.
Then $f^{-1}(c)$ is reduced (resp.\,irreducible and the multiplicity of its generic point is $l$) if $c\neq c_{1}$ (resp.\,$c=c_{1}$).
To see that the sequence (\ref{exac}) is exact, it suffices to show that condition 4 in Proposition \ref{essential} is satisfied.
Let $X'$ be a connected \'etale covering space of $X(=Z)$.
Then the normalization of $C$ in the function field of $X'$ is \'etale over $C\setminus\{c\}$, which is a $\Z/N\Z$-Galois \'etale covering for some $N\in\N$ since $C\setminus\{c\}$ is isomorphic to $\mathbb{G}_{m,k}$.
By considering the multiplicities of the fibers of $f$, we have $N=1$ or $l$. 
On the other hand, since $X'$ does not factor through $C'\times_{\Spec k}C' \to X$ by Lemma \ref{interfield}.2, $l$ does not divide $N$.
Therefore, $N=1$ and condition 4 in Proposition \ref{essential} is satisfied.
\end{enumerate}
\label{norreg}
\end{exam}

\label{necsection}

\section{An application to morphisms to curves}
In this section, we apply Proposition \ref{geomexact} to morphisms from smooth varieties to smooth curves over a field of characteristic $0$.

\begin{dfn} (\cite{Ho2}\,Definition 2.5)
We shall write
$$\mathbb{P}_{\not\exists \twoheadrightarrow \infty}$$
for the property of a profinite group defined as follows: A profinite group $G$ has property $\mathbb{P}_{\not\exists \twoheadrightarrow \infty}$ if, for an arbitrary open subgroup $H$ of $G$, there exists a prime number $l_{H}$ such that there are no quotient profinite groups of $H$ which are free pro-$l_{H}$ and not topologically finitely generated.
\end{dfn}

Let $k$ be a field of characteristic $0$, $S$ a smooth curve over $k$, $X$ a normal, separated scheme of finite type and geometrically connected over $k$, and $f$ a dominant morphism from $X$ to $S$ over $k$.
Write $N_{X/S}$ for the normalization of $S$ in the function field of $X$ and $S'\to S$ for the maximal \'etale subextension of $N_{X/S}\to S$.
Then we have a natural factorization $X\to N_{X/S}\to S' \to S$.
Write $f'$ for the morphism $X\to S'$
Let $\overline{\eta'}$ be a geometric generic point of $S'$.
Write $X_{\overline{\eta'}}$ for the scheme $X\times_{S'}\overline{\eta'}$.
Take a geometric point $\overline{x}$ of $X_{\overline{\eta'}}$.

\begin{thm} (cf.\,Theorem \ref{intro curve criterion})
Consider the following conditions:
\begin{enumerate}
\item
The morphism $f'$ is surjective and the scheme $X_{\overline{\eta'}}$ is connected. Moreover, the greatest common divisor of the multiplicities of the irreducible components of each closed fiber of $f'$ is $1$.
\item The scheme $X_{\overline{\eta'}}$ is connected and the sequence of the \'etale fundamental groups
\begin{equation}
\pi_{1}(X_{\overline{\eta'}}, \overline{x})\rightarrow \pi_{1}(X, \overline{x}) \rightarrow \pi_{1}(S', \overline{x}) \rightarrow 1
\label{exact}
\end{equation}
is exact.
\item
The profinite group $\mathrm{Ker}(\pi_{1}(X, \overline{x})\rightarrow \pi_{1}(S, \overline{x}))$ has property $\mathbb{P}_{\not\exists \twoheadrightarrow \infty}$.
\end{enumerate}
Then it holds that $1\Rightarrow 2 \Rightarrow 3$.
If the scheme $S$ is neither a proper rational curve nor an affine line, it holds that $3 \Rightarrow 2$.
If the scheme $S$ is not proper rational and $X$ is regular, it holds that $2 \Rightarrow 1$.
\label{curve criterion}
\end{thm}

\begin{proof}
The implications between conditions $2$ and $3$ are results of Hoshi (cf.\,\cite{Ho2} Theorem 2.8).
Note that $X_{\overline{\eta'}}$ is connected if and only if $N_{X/S}=S'$.
Then, $1\Rightarrow 2$ follows from Proposition \ref{geomexact}.

Suppose that condition 2 is satisfied and the scheme $S$ is not proper rational.
Then $S'$ is also not proper rational.
Note that, if $f'$ is surjective, condition 1 is satisfied by Proposition \ref{geomexact}.
Hence, to finish the proof of Theorem \ref{curve criterion}, it suffices to show that the morphism $f'$ is surjective.
Suppose that $f'$ is not surjective.
Since $S'$ satisfies property (T) by Lemma \ref{nonproper}, there exists a connected \'etale covering space $X'\to X$ such that the normalization of $S'$ in the function field of $X'$ is not \'etale over $S'$.
This contradicts the assumption that the sequence (\ref{exact}) is exact by Proposition \ref{essential} and Remark \ref{dominant}.
\end{proof}

\begin{rem}
By \cite{Ho2} Remark 2.5.1, a topologically finitely generated profinite group satisfies property $\mathbb{P}_{\not\exists \twoheadrightarrow \infty}$.
Suppose that $S$ is a hyperbolic curve over $k$ and $X$ is regular.
Since the profinite group $\pi_{1}(X_{\overline{\eta'}}, \overline{x})$ is topologically finitely generated, the conditions in Theorem \ref{curve criterion} hold if and only if the profinite group $\mathrm{Ker}(\pi_{1}(X, \overline{x})\rightarrow \pi_{1}(S, \overline{x}))$ is topologically finitely generated.
\end{rem}

\begin{rem}
If we drop the assumption that the scheme $X$ is regular, the implication $2\Rightarrow1$ does not hold in general (cf.\,Example \ref{norreg}.2).
\end{rem}

\label{curves}

\section{Appendix 1: property (F)}
In this section, we discuss property (F) (cf.\,Definition \ref{(F)}).
\label{app}
\subsection{Examples}
If we drop property (F), Theorem \ref{suff} does not hold in general.

\begin{exam}
Let $K$ be a strictly henselian discrete valuation field with imperfect residual field $k$ of characteristic $p>0$.
Write $O_{K}$ for the valuation ring of $K$.
Let $\mathfrak{C} \rightarrow \Spec O_{K}$ be a proper smooth morphism of relative dimension $1$ with geometrically connected fibers.
Suppose that there exists a $\Z/p\Z$-Galois \'etale covering space $\mathfrak{X}\to\mathfrak{C}$.
Choose a generator $\sigma$ of the Galois group $\Z/p\Z\, ( \subset \mathrm{Aut}(\mathfrak{X}))$.
Let $K'\supset K$ be a $\Z/p\Z$-Galois extension whose residual extension is purely inseparable of degree $p$.
Write $O_{K'}$ (resp.\,$k'$) for the valuation ring of $K'$ (resp.\,the residual field of $O_{K'}$).
Choose a generator $\tau$ of the Galois group $\Z/p\Z\, ( \subset \mathrm{Aut}(\Spec K'))$ and consider a $\Z/p\Z$-action on the scheme $\mathfrak{X}\times_{\Spec O_{K}}\Spec O_{K'}$ induced by the automorphism $(\sigma,\tau)$.
Then the second projection $\mathfrak{X}\times_{\Spec O_{K}}\Spec O_{K'}\rightarrow \Spec O_{K'}$ is a $\Z/p\Z$-equivariant morphism.
\begin{equation}
\xymatrix{ 
\mathfrak{X}\times_{\Spec O_{K}}\Spec O_{K'} \ar[d] \ar[r] & \Spec O_{K'} \ar[d] 
&\mathfrak{X}\times_{\Spec O_{K}}\Spec O_{K'} \ar[d] \ar[r] & \Spec O_{K'} \ar[d] \\
\mathfrak{X} \ar[r] & \Spec O_{K} 
& \mathfrak{Z} \ar[r] & \Spec O_{K} 
}
\label{(F)bad}
\end{equation}
Write $\mathfrak{Z}$ for the quotient scheme $(\mathfrak{X}\times_{\Spec O_{K}}\Spec O_{K'})/\langle\sigma\times\tau\rangle$.
$\mathfrak{Z}$ is a scheme over $\Spec O_{K}$ and its special fiber is isomorphic to $\mathfrak{C}\times_{\Spec O_{K}}\Spec k'$ over $k$.
Since the natural morphism $\mathfrak{X}\times_{\Spec O_{K}}\Spec O_{K'} \rightarrow \mathfrak{Z}$ is finite \'etale, the scheme $\mathfrak{Z}$ is regular.
Note that, in the diagram (\ref{(F)bad}), the left square is Cartesian and the right square is not Cartesian.
The normalization of $\Spec O_{K}$ in the function field of $\mathfrak{X}\times_{\Spec O_{K}}\Spec O_{K'}$ coincides with $\Spec O_{K'}$.
Therefore, if we consider the case where $X\to S$ (in Section \ref{suffsec}) is $\mathfrak{Z} \rightarrow \Spec O_{K}$, the sequence (\ref{exac}) is not exact by Proposition \ref{essential}.
Note that the greatest common divisor of the multiplicities of the irreducible components of the special fibers is $1$.
\label{curve(F)}
\end{exam}

\begin{rem}
\begin{enumerate}
\item
We do not need to assume that $\mathfrak{C}$ is of relative dimension $1$ over $ \Spec O_{K}$  in the discussion given in Example \ref{curve(F)}.
\item
If we replace the condition on residual extension of $K'\supset K$ in Example \ref{curve(F)} with the condition that the ramification index of the extension $K' \supset K$ is $p$, the multiplicity of the (unique) irreducible component of the special fiber is $p$.
Therefore, we need to suppose that the greatest common divisor in property (R) is not divisible by $p$.
\end{enumerate}
\label{curve(F)rem}
\end{rem}

\subsection{Generalities on (F)}
In this subsection, we discuss generalities on (F).
Let $k$ be a field and $\iota_{i}: k \hookrightarrow K_{i}\;(i\in I)$ inclusions of fields.

\begin{prop}
Write $k_{i}$ for the algebraic closure of $k$ in $K_{i}$ and $\iota_{i}' : k \hookrightarrow k_{i}$ for the inclusion induced by $\iota_{i}$.
Moreover, write $k_{i}^{\mathrm{sep}}$ for the (absolute) separable closure of $k_{i}$ and $ \iota_{i}^{\mathrm{sep}} : k \hookrightarrow k_{i}^{\sep}$ for the inclusion induced by $\iota_{i}'$.
The following are equivalent:
\begin{enumerate}
\item
The inclusions $\{\iota_{i}\}$ satisfy property (F).
\item The inclusions $\{ \iota_{i}'\}$ satisfy property (F).
\item The inclusions $\{ \iota_{i}^{\mathrm{sep}}\}$ satisfy property (F).
\end{enumerate}
\label{alg}
\end{prop}

\begin{proof}
Proposition \ref{alg} follows from the definition of property (F) and elementary field theory.
\end{proof}

\begin{dfn}
We say that the inclusions $\{ \iota_{i} \}$ satisfy property (F') if the following condition is satisfied:\\
(F'): For any subfield $l$ of the product ring $\underset{1\leq i \leq n}{\prod}K_{i}$ which is algebraic over the diagonal subfield $k$ defined by $\iota_{i}$, the extension $k \subset l$ is separable.
\label{(F')}
\end{dfn}

\begin{exam}
Property (F) implies property (F'), but property (F') does not imply property (F).
Consider the inclusions $\F_{p}(X^{p}+Y^{p},X^{p}Y^{p})\hookrightarrow \F_{p}(X,Y^{p})$ and $\F_{p}(X^{p}+Y^{p},X^{p}Y^{p})\hookrightarrow \F_{p}(X+Y, XY)$.
These inclusions satisfy property (F').
On the other hand, the field extension $\F_{p}(X+Y, XY) \subset \F_{p}(X,Y)$ is separable and the field $\F_{p}(X,Y)$ contains the field $\F_{p}(X,Y^{p})$ which is inseparable over the field $\F_{p}(X^{p}+Y^{p},X^{p}Y^{p})$.
\end{exam}

\begin{lem}
Let $k \subset K'$ be an extension of fields.
Write $k'$ for the algebraic closure of $k$ in $K'$,  $k'_{s}$ (resp.\,$\overline{k'}$) for a(n absolute) separable closure of $k'$ (resp.\,an (absolute) algebraic closure of $k'_{s}$), and $k_{s}$ for the separable closure of $k$ in $\overline{k'}$.
Moreover, write $k'_{n}$ (resp.\,$k'_{s,n}$) for the normal closure of $k'$ over $k$ in $\overline{k}$ (resp.\,the normal closure of $k'_{s}$ over $k$ in $\overline{k}$).
Furthermore, write $k'_{p}$ (resp.\,$k'_{s,p}$) for the field $k'_{n}\cap k^{p^{-\infty}}$ (resp.\,the field $k'_{s,n}\cap k^{p^{-\infty}}$).
Then the following hold:
\begin{enumerate}
\item
$k'_{p}=k'_{s,p}$.
\item
$k_{s}$ and $k'_{s,p}$ are linearly disjoint over $k$ and we have $k'_{s,n}=k_{s}k'_{s,p}$.
\end{enumerate}
\label{normalvs}
\end{lem}

\begin{proof}
Note that we have $k'_{s}=k'k_{s}$.
Since $k'_{n}$ is the composite field of the separable closure of $k$ in $k'_{n}$ and $k'_{p}$, we have
$$k'_{s}\subset k'_{n}k_{s}=k'_{p}k_{s}\subset k'_{s,n}.$$
Since $k'_{n}$ and $k_{s}$ are normal over $k$, we have $k'_{p}k_{s}=k'_{s,n}$.
Since $k'_{p}$ and $k_{s}$ are linearly independent over $k$, Lemma \ref{normalvs} holds.
\end{proof}

\begin{rem}
We obtain an extension field $k_{i,p}$ of $k$ for each $i\in I$ as $k'_{p}$ by replacing $k\subset K'$ in Lemma \ref{normalvs} with $k\subset K_{i}$.
Note that we cannot consider the intersection of $k_{i} \;(i\in I)$, although we can consider the intersection of $k_{i,p} \;(i\in I)$.
\end{rem}

\begin{thm}
The inclusions $\{\iota_{i}\}$ satisfy property (F) if the intersection of $k_{i,p} \,(i\in I)$ coincides with $k$.
\end{thm}

\begin{proof}
This follows from Proposition \ref{alg} and Lemma \ref{normalvs}.
\end{proof}

\begin{prop}
Suppose that the degree of extension $k^{p} \subset k$ is $\leq 1$.
\begin{enumerate}
\item Any algebraic extension of $k$ is a linear disjoint sum of an algebraic separable extension of $k$ and a purely inseparable extension of $k$.
\item
The inclusions $\{\iota_{i}\}$  satisfy property (F) if and only if the algebraic closure of $k$ in each $k_{i}$ is separable over $k$.
\end{enumerate}
\end{prop}

\begin{proof}
Assertion 2 follows from assertion 1.
Let $M$ (resp.\,$M_{i}$; $M_{s}$) be an algebraic extension field of $K$ (resp.\,the purely inseparable closure of $K$ in $M$; the separable closure of $K$ in $M$).
We assume that $M$ is finite over $K$ and show that $[M:M_{s}]=[M_{i}:K]$, from which assertion 1 follows.
Since we have $[M_{s}^{1/p}:M_{s}]=[K^{1/p}:K]\leq p$, $M$ is inseparable over $K$ if and only if $K^{1/p}\subset M$.
Then $[M:M_{s}]=[M_{i}:K]$ follows from induction on $[M:M_{s}]$.
\end{proof}

\begin{exam}
We give some examples of fields $k$ such that the degree of extension $k^{p} \subset k$ is $\leq 1$.
\begin{enumerate}
\item A perfect field.
\item An extension field of a perfect field with transcendental degree $1$.
\item A field of Laurent series over a perfect field.
\end{enumerate}
\end{exam}

\section{Appendix 2: geometrically connected fibers}
In this section, we discuss the homotopy exact sequence (\ref{exac}) in the case where the geometric point $\overline{\eta}$ is not necessarily over the generic point of $S$.

Let $X, S, f$, $\overline{\eta}$, and $\overline{x}$ be as in Section \ref{suffsec}.
Consider a geometric (not necessarily generic) point $\overline{\eta'}$ of $S$.
Write $\widetilde{S}_{\overline{\eta'}}$ for the strict localization of $S$ at $\overline{\eta'}$ and fix an $S$-morphism $\overline{\eta} \rightarrow \widetilde{S}_{\overline{\eta'}}$.

\begin{rem}
\begin{enumerate}
\item If the sequence (\ref{exac}) is exact, the sequence
\begin{equation}
\pi_{1}(X\times_{S}\widetilde{S}_{\overline{\eta'}}, \overline{x})\rightarrow \pi_{1}(X, \overline{x}) \rightarrow \pi_{1}(S, \overline{x}) \rightarrow 1
\label{exachen}
\end{equation}
is exact.
\item
Suppose $f$ satisfies the assumptions other than condition (R) in Theorem \ref{suff} and instead satisfies the following condition (R'):\\
(R'): Let $s$ be a point of $S\setminus\mIm(\overline{\eta'})$ whose local ring is of dimension $1$.
Write $\xi_{i}\;(i\in I)$ for the generic points of the scheme $f^{-1}(s)$, $e_{i}$ for the multiplicity of $\xi_{i}$, and $k(\xi_{i})$ (resp.\,$k(s)$) for the residual field of $\xi$ (resp.\,$s$).
Then $\mathrm{gcd}_{i\in I}e_{i}=1$ and the inclusions $k(s) \hookrightarrow k(\xi_{i})\;(i\in I)$ satisfy property (F).\\
Then, by using an argument similar to that given in Section \ref{suffsec}, we can show that the sequence (\ref{exachen}) is exact.
\item
From Remarks \ref{rem1}.1 and  \ref{rem1}.2, the homomorphism $\pi_{1}(X\times_{S}\overline{\eta}, \overline{x}) \to \pi_{1}(X\times_{S}\widetilde{S}_{\overline{\eta'}}, \overline{x})$ is not surjective in general.
\end{enumerate}
\label{rem1}
\end{rem} 

Take a geometric point $\overline{x'}$ of $X\times_{S}\overline{\eta'}$.

\begin{rem}
\begin{enumerate}
\item
Suppose that the morphism $X\times_{S}\widetilde{S}_{\overline{\eta'}} \rightarrow \widetilde{S}_{\overline{\eta'}}$ is proper and flat.
Note that $\Spec (f\times \mathrm{id}_{\widetilde{S}_{\overline{\eta'}}})_{\ast}O_{X\times_{S}\widetilde{S}_{\overline{\eta'}}}$ is connected and normal and the morphism
$$\pi:\Spec (f\times \mathrm{id}_{\widetilde{S}_{\overline{\eta'}}})_{\ast}O_{X\times_{S}\widetilde{S}_{\overline{\eta'}}}\to \widetilde{S}_{\overline{\eta'}}$$
is finite.
Since $X\times_{S}\overline{\eta}$ is connected, $\pi$ is a universally homeomorphism and therefore the scheme $X\times_{S}\overline{\eta'}$ is connected.
Then the homomorphism $\iota: \pi_{1}(X\times_{S}\overline{\eta'}, \overline{x'}) \rightarrow\pi_{1}(X\times_{S}\widetilde{S}_{\overline{\eta'}}, \overline{x'})$ is an isomorphism.
\item
$\iota$ in Remark \ref{rem2}.1 is neither surjective nor injective in general.
\end{enumerate}
\label{rem2}
\end{rem}

\begin{cor}
Suppose that $f$ is proper and flat (cf.\,Remark \ref{rem2}.1).
Then the sequence
\begin{equation}
\pi_{1}(X\times_{S}\overline{\eta'}, \overline{x'})\rightarrow \pi_{1}(X, \overline{x'}) \rightarrow \pi_{1}(S, \overline{x'}) \rightarrow 1
\label{exaccl}
\end{equation}
is exact if condition (R') is satisfied.
\label{appcor}
\end{cor}

\begin{proof}
Corollary \ref{appcor} follows from Remark \ref{rem1}.2 and Remark \ref{rem2}.1.
\end{proof}

Suppose that $X\times_{S}\overline{\eta'}$ is connected.
We need an ad hoc assumption (cf.\,the third condition in Proposition \ref{adhoc}) to make the sequence (\ref{exaccl}) exact by Remark \ref{rem1}.3 and Remark \ref{rem2}.2.

\begin{prop} (cf.\,\cite{Ho} Proposition 1.10)
Suppose that the following conditions are satisfied:
\begin{itemize}
\item
The morphism $f$ is flat or the scheme $S$ is regular.
\item
$f$ satisfies property (R').
\item
Let $X' \to X$ be a connected finite \'etale covering space.
Write $K_{X'/S}$ for the algebraic separable closure of the function field of $S$ in the function field of $X'$ and $N_{X'/S}$ for the normalization of $S$ in $K_{X'/S}$.
For any $S$-morphism $\overline{\eta'} \rightarrow N_{X'/S}$, the scheme $\overline{\eta'}\times_{N_{X'/S}}X'$ is connected.
\end{itemize}
Then the sequence (\ref{exaccl}) is exact.
\label{adhoc}
\end{prop}

\begin{proof}
It suffices to show that the implication $4 \Rightarrow 3$ in Proposition \ref{essential} holds if we replace $\overline{\eta}$ with $\overline{\eta'}$.
Let $X'$ be as in condition 3 in Proposition 1.5.
Since the number of the connected components of the scheme $X_{\overline{\eta'}}\times_{X}X' =\overline{\eta'}\times_{S}X'=(\overline{\eta'}\times_{S}N_{X'/S})\times_{N_{X'/S}}X'$ coincides with the covering degree of $N_{X'/S} \rightarrow S$ by the third condition, the assertion holds.
\end{proof}

\label{app2}

\end{document}